\documentclass[11pt]{amsart}
\usepackage[a4paper,centering,left=2.4cm,right=2.4cm]{geometry}
\usepackage{amsfonts,amscd,amssymb,amsmath,amsthm,hyperref,mathrsfs,multirow,xcolor,lscape,longtable,pbox,lipsum,afterpage,graphicx,threeparttable}
\usepackage[thinlines]{easytable}
\usepackage{microtype}

\usepackage{tikz-cd}

\usepackage[all,cmtip]{xy}
\usepackage{tikz}
\usetikzlibrary{arrows}
\usetikzlibrary{matrix}
\usetikzlibrary{calc}

\newtheorem{theorem}{Theorem}[subsection]
\newtheorem{definition}{Definition}[subsection]

\newtheorem{lemma}{Lemma}[subsection]
\usepackage[OT2,T1]{fontenc}
\DeclareSymbolFont{cyrletters}{OT2}{wncyr}{m}{n}
\DeclareMathSymbol{\Sha}{\mathalpha}{cyrletters}{"58}
\newtheorem{claim}{Claim}[subsection]
\newtheorem{proposition}{Proposition}[subsection]

\newtheorem{corollary}{Corollary}[subsection]

\newtheorem{remark}{Remark}[subsection]
\newtheorem{conjecture}{Conjecture}

\theoremstyle{definition}
 \numberwithin{equation}{section}
\theoremstyle{definition}
\theoremstyle{remark}
 \numberwithin{equation}{section}

\renewcommand{\leq}{\leqslant}
\renewcommand{\geq}{\geqslant}

\usepackage[
  locale = DE % comma as decimal mark
]{siunitx}

\renewcommand{\and}{\quad \mbox{and} \quad}  %% "and" (text in equation)
\renewcommand{\leq}{\leqslant}
\renewcommand{\geq}{\geqslant}

\begin{document}

\title{Distribution of Non-Wieferich primes in certain algebraic groups under ABC}
\date{\today}
\author{Subham Bhakta}
\address{Mathematisches Institut, Bunsenstrasse 3-5, D-37073 Germany}
\email{subham.bhakta@mathematik.uni-goettingen.de}

%\author{????}
%\address{????}
%\email{???}

\begin{abstract} Under ABC, Silverman showed that there are infinitely many non-Wieferich primes with respect to any (non-trivial) base $a$. Recently Srinivas and Subramani proved an analogous result over number fields with trivial class group. In the first part of this article, we extend their result to any arbitrary number fields. Secondly, we give an asymptotic lower bound for the number of non-Wieferich prime ideals. Furthermore, we show a lower bound of same order is achievable for non-Wieferich prime ideals having norm congruent to $1 \pmod k.$ Lastly, we generalize Silverman's work for elliptic curves over arbitrary number fields following the treatment by K\"uhn and M\"uller.
\end{abstract}

\maketitle 

\section{Introduction}
Classically a rational prime $p$ is called non-Wieferich prime with respect to base $a$, if 
$$a^{p-1} \equiv 1 \pmod p \hspace{0.15cm} \text{and} \hspace{0.15cm} a^{p-1} \not\equiv 1 \pmod {p^2},$$ holds simultaneously. It is not known whether there are infinitely many non-Wieferich primes or not. Under ABC conjecture, it is known that there are infinitely many non-Wieferich primes with (non-trivial) base $a$. By non-trivial we mean $a\neq \pm 1.$ Silverman showed that there are at least $c\log x$ many non Wieferich primes up to $x$, for some constant $c>0,$ depending on base $a.$ A number field analog of the same problem was considered by K. Srinivas and M. Subramani in \cite{Srini}. In the preliminary section, we shall recall their notations and main results, which are the starting point of our article. For number fields with class number one, they showed that there are infinitely many non-Wieferich primes with base $\eta,$ under some certain conditions on $\eta$ and assuming ABC conjecture over number fields. In this article, we first extend their result to any arbitrary number field with relaxing some conditions on the unit $\eta.$ More precisely, in section 3 we prove the following 
\begin{theorem} \label{Thm 1} Suppose $K$ be any arbitrary number field and assume ABC holds over $K .$ Let $\eta$ be a unit such that $|\sigma(\eta)|<1$ for all but exactly one embedding $\sigma.$ Then there are infinitely many non-Wieferich prime ideals with respect to $\eta.$

\end{theorem} 

\noindent In the same section, we give an asymptotic lower bound generalizing Silverman's ideas over number fields. In other words, we prove

\begin{theorem} \label{Thm 2} Under the same assumption of the previous theorem, there are at least $c\log x$ many non Wieferich prime ideals with respect to $\eta$ of norm at most $x,$ where the $c>0$ is a constant depending on $K$ and the unit $\eta.$ 
\end{theorem}
\noindent In section $4$ we discuss non-Wieferich primes in certain congruence classes. Under ABC, Graves-Murty (see \cite{Murty}) first showed that there are at least $c \frac{\log x}{\log \log x}$ many non Wieferich primes up to $x.$ Later Ding and Chen in [6], generalized the result by showing existence of at least $c \frac{\log x}{\log \log x}(\log\log x)^M$ many primes up to $x$, for any positive integer $M.$ We generalize both of these results to arbitrary number fields with an asymptotic lower bound of order $c\log x.$ More precisely, we prove
\begin{theorem} \label{Thm 3}Let $k$ be a fixed natural number. Under the same assumption of Theorem 1.0.1, there are at least $c\log x$ many non Wieferich prime ideals $\pi$ with respect to $\eta,$ such that $N(\pi) \leq x$ and $N(\pi) \equiv 1 \pmod k,$ where the constant $c>0$ depends only on $K$ and unit $\eta.$

\end{theorem}
\noindent In the last section, we discuss a more general problem. Suppose $G$ be a commutative algebraic group, and $P \in G(K)$ be a point of infinite order. An analogous problem in this generalized situation asks whether $N_{\mathfrak{p}}P \equiv 1 \pmod {\mathfrak{p}^2},$ where $N_{\mathfrak{p}}=|G(\mathbb{F}_{\mathfrak{p}})|.$ For instance when we take $G$ to be the multiplicative group $\mathbb{G}_m$, and $P \in \mathbb{G}(K)$ to be a \textit{non-torsion} unit, the problem then asks about order of the non-unit when reduced modulo $\mathfrak{p}^2.$ In other words, our section 3 and 4 are devoted for $\mathbb{G}_m.$ Silverman studied this general problem over elliptic curves. He showed that, under ABC there are infinitely many (in fact an asymptotic lower bound of order $c\sqrt{\log x}$) non-Wieferich primes for elliptic curves with $j$ invariant $0$ and $1728$. The same arguments could not be applied to other elliptic curves due to the unavailability of an inequality involving the height of points. This case was later settled by K\"uhn and M\"uller \cite{Kuhn}. We shall discuss their result and prove the following
\begin{theorem} Let $E$ be an elliptic curve defined over an arbitrary number field $K$ and $P \in E(K)$ be a point of infinite order. If the ABC conjecture over $K$ is true, then are at least $c \sqrt{\log x}$ many non-Wieferich primes with respect to $P$ of norm at most $x.$ Here $c>0$ be a constant, depending only on $K$ and $P.$

\end{theorem}

\section{preliminary} 
\subsection{Notations} Let $K$ be a number field, and $\mathcal{O}_K$ be its ring of integers. Denote $M_K$ to be the (up-to equivalence) set of valuations in $K$. We further consider $M_{K, \infty}$ as the set of archimdean places, and $M'_K$ as the set of finite places of $K,$ up-to equivalence. For any $\alpha \in \mathcal{O}_K$ and $\mathfrak{p}\in M'_K,$ we denote 
\[ ||\alpha||_{\mathfrak{p}}= N_{K/\mathbb{Q}}(\mathfrak{p})^{-v_{\mathfrak{p}}(\alpha)},\]
where $N_{K/\mathbb{Q}}(\mathfrak{p})$ is the index of $\mathfrak{p}$ in $\mathcal{O}_K,$ and $v_{\mathfrak{p}}(\alpha)$ is the maximum power of $\mathfrak{p}$ dividing $\alpha.$ For an element $\alpha \in \mathcal{O}_K$ we denote 
$$N_{K/\mathbb{Q}}(\alpha)=\prod_{\sigma \in M_K} \sigma(\alpha).$$
Throughout this article we shall denote $N_{K/\mathbb{Q}}(I) := N(I)$ for any ideal $I$ in $\mathcal{O}_K.$
\begin{definition}[Weil height]
Let $\alpha \in \mathbb{P}^1(K)$, then the Weil height of $\alpha$ is defined to be,
\[H(\alpha)=\Big(\prod_{v \in M_K}  \max \{||\alpha||_v,1\}\Big)^{\frac{1}{[K:\mathbb{Q}]}}.\]
In particular, for a point $P=[a:c] \in \mathbb{P}^1(K)$ we have
\[H(P)=\Big(\prod_{v \in M_K} \max\{||a||_v, ||c||_v\}\Big)^{\frac{1}{[K:\mathbb{Q}]}}.\]
Moreover, for a triple $a,b,c \in K$ we define
\[H(a,b,c)=\Big(\prod_{v \in M_K} \max\{||a||_v, ||b||_v, ||c||_v\}\Big)^{\frac{1}{[K:\mathbb{Q}]}}.\]
\end{definition}

%\begin{remark} We are omitting details, but one can see the definitions of $H(.)$ and $\text{rad}(.)$ are well defined, in other words they are invariant under considering $a,b$ or $c$ in some bigger extension. Moreover the inequality $H(a,c,a+c) \ll_{K} H([a:c])$ holds true for any $a,c \in K.$
%\end{remark}

\subsection{ABC over number fields} Let $a,b,c$ be elements in $K$ such that $a+b+c=0.$ Define the radical of an element $\alpha \in \mathcal{O}_K$ by
\[\text{rad}(\alpha)= \Big(\prod_{\mathfrak{p} \hspace{0.07cm} \text{prime in} \hspace{0.1cm} \mathcal{O}_K, \hspace{0.07cm}\mathfrak{p} \mid \alpha} N_{K/\mathbb{Q}}(\mathfrak{p})^{v_{\mathfrak{p}}(p)}\Big)^{\frac{1}{[K:\mathbb{Q}]}}\]

\begin{conjecture}[ABC over number fields]\label{ABC} Let $a,b,c$ be a triple in $K$ such that $a+b+c=0$. Then for any $\epsilon>0,$
\[H(a,b,c) \ll_{\epsilon, K} \big(\mathrm{rad}(abc)\big)^{1+\epsilon}.\]
\end{conjecture}

\begin{remark} We are omitting details, but one can see the definitions of $H(.)$ and $\mathrm{rad}(.)$ are well-defined, in other words they are invariant under considering $a,b$ or $c$ in some bigger extension. Moreover the inequality $H(a,c,a+c) \ll_{K} H([a:c])$ holds true for any $a,c \in K.$
\end{remark}

\subsection{Non Wieferich primes and units} A prime ideal $\pi$ of $\mathcal{O}_K$ is said to be non-Wieferich prime with respect to base $\eta$, if the following
\[\eta^{N(\pi)-1} \equiv 1 \pmod \pi, \hspace{0.1cm} \eta^{N(\pi)-1} \not\equiv \pi^2 \pmod \pi\] holds simultaneously. It is known that $\mathcal{O}^{*}_K$, the group of units of $\mathcal{O}_K$ has rank $r+s-1$ as a $\mathbb{Z}$-module, where $r$ is the number of real embeddings and $s$ the number of conjugate pairs of complex embeddings of $K.$ Some of the theorems which we wish to prove are about certain units. These units have absolute value less than for one all but exactly one embedding. So we need to ensure existence of such a unit for our sake. For which we can use the following lemma from \cite{Ram}
\begin{lemma} \label{Ram book} Let $\sigma \in M_K$ be a real archimedian place of $K.$ Then there exist a unit $\eta \in \mathcal{O}_K$ such that,
\[|\sigma(\eta)|>1, |\tau(\eta)|<1 \hspace{0.1cm}\text{for all} \hspace{0.15cm}\tau \in M_{K, \infty}-\{\sigma\}.\]
\end{lemma}
\noindent We can say even more,
\begin{proposition} Density of the units satisfying the condition of Lemma \ref{Ram book} is
\[\frac{r+s-1}{2^{r+s-1}}.\]
\end{proposition}
\begin{proof} Consider the embedding used in the proof of Dirichlet's unit theorem,
\[\mathcal{O}^*_K \hookrightarrow \mathbb{R}^{r+s-1}\]
given by $\eta \mapsto \Big(\log\big(\sigma_i(\eta)\big)\Big)_{1\leq i\leq r+s-1}.$ So the problem about counting,
\[\{\eta \in \mathcal{O}^{*}_K \mid H(\eta) \leq x, |\sigma(\eta)|<1 \hspace{0.1cm} \text{for all but one embedding}.\}\]
It is now same as estimating,
\[\{(x_1,x_2,\cdots, x_{r+s-1}) \in \mathbb{Z}^{r+s-1} \mid |x_i| \leq \log x, \hspace{0.1cm} x_i<0 \hspace{0.1cm}\text{for all but one coordinate}\}.\]
The above quantity is about $c_K (r+s-1)(\log X)^{r+s-1}.$ On the other hand, total number of units of height at most $x$ is $c_K2^{r+s-1}(\log x)^{r+s-1},$ where $c_K$ is co-volume of $\mathcal{O}_K^{*}$. And so does the result follows.
\end{proof}
\noindent The reason behind mentioning the result is that, we are working over a certain set of units, and the proposition above shows that the set has positive density. So our domain is not that bad. One may also ask why are we not considering non units ? Following the proofs in the next section, one can see there is really no problem with arbitrary algebraic integers once we have property like Lemma 2.3.1. But the following proposition says, that phenomenon is not very likely to happen. 

%Pollack also notes that not every number field has a finite extension with class number 1. He cites a theorem of Golod and Shafarevich, "On the class field tower," (Russian) Izv. Akad. Nauk SSSR Ser. Mat. 28 (1964) 261-272, to the effect that for every $n$ there are infinitely many number fields $K$ of degree $n$ which do not have a finite extension of class number $1$. He gives (without proof) the example $\mathbb{Q}(\sqrt{-3.5.7})$/ as a quadratic field with no finite extension of class number 1

\begin{proposition} For any arbitrary number fields, set of $\alpha \in \mathcal{O}_K$ such that $|\sigma(\alpha)|<1$ for all but one embedding, have zero density. 
\end{proposition}
\begin{proof} The map,
\[\alpha \mapsto \Big(\sigma(\alpha)_{\sigma \in M_K}\Big)\]
embeds $O_K$ into $\mathbb{R}^{\deg(K)}$ as a lattice of dimension $\deg(K).$ We have
\[\{\alpha \in \mathcal{O}_K \mid H(\alpha) \leq x\} \sim C_K x^{\deg(K)}\]
and also,
\[\{\alpha \in \mathcal{O}_K \mid H(\alpha) \leq x, |\sigma(\alpha)|<1 \hspace{0.1cm}\text{for all but one embedding}\}\sim C_K x.\]
Indeed, the desired set has zero density. 

\end{proof}
\begin{remark} The constant $c_K$ above is co-volume of the lattice $\mathcal{O}_K$ in $\mathbb{R}^{\deg(K)}.$ It is well know the co-volume is explicitly given by $\mathrm{disc}(\mathcal{O}_K).$ See page 7 in \cite{Shan} for more details.

\end{remark}
\subsection{Nuts and bolts to fix problem concerning class number}

K. Srinivas and M. Subramani proved their result in the case of class number one, under some certain conditions on unit $\eta.$ The assumption on class number was required to write some elements as a product of primes uniquely (see their proof of Lemma 4.1 in \cite{Srini}). Class number one really guarantees this because, $\mathcal{O}_K$ then becomes a PID, and a UFD as well. In general, it is possible that $K$ does not have class number one, while some of its extension have. Then we could then do all of these over that extension and still get our job done, which is illustrated by the next lemma.
\begin{lemma} \label{lifting} Let $L$ be an extension of $K$ and $\mathfrak{P}$ be a prime ideal of $\mathcal{O}_L$ lying over a prime ideal $\mathfrak{p}$ in $\mathcal{O}_K$. If $\mathfrak{P}$ is a non-Wieferich prime in $\mathcal{O}_L$ with respect to a unit $\eta \in \mathcal{O}_K$, then $\mathfrak{p}$ is a non-Wieferich prime in $K$ with respect to $\eta$. Conversely, if $\mathfrak{p}$ splits completely in $L/K,$ and is a non-Wieferich prime with respect to a unit $\eta \in \mathcal{O}_K,$ then $\mathfrak{P}$ is non-Wieferich prime with respect to $\eta$ for any prime $\mathfrak{P}$ in $\mathcal{O}_L$ lying over $\mathfrak{p}$.

\end{lemma}

\begin{proof} Let $\mathfrak{P}$ be a prime ideal non-Wieferich prime in $\mathcal{O}_L$ with respect to a unit $\eta \in \mathcal{O}_K,$ and $\mathfrak{p}$ be a prime in $\mathcal{O}_K$ lying below it. For sake of contradiction, suppose $\mathfrak{p}$ is Wieferich in $\mathcal{O}_K$ with respect to $\eta$. Then
\[\eta^{N_{K/\mathbb{Q}}(\mathfrak{p})} \equiv 1 \pmod {\mathfrak{p}^2} .\]
Now $N_{L/\mathbb{Q}}(\mathfrak{P})$ is some power of $N_{K/\mathbb{Q}}(\mathfrak{p})$, and so
\[\eta^{N_{L/\mathbb{Q}}(\mathfrak{P})} \equiv 1 \pmod {\mathfrak{P}^2},\]
contradicting our assumption that $\mathfrak{p}$ is Wieferich in $\mathcal{O}_K$.\\
\newline
\noindent For the converse, if $\mathfrak{P}$ is non-Wieferich with respect to base $\eta$, then
\[\eta^{N_{L/\mathbb{Q}}(\mathfrak{P})} \not\equiv 1 \pmod{\mathfrak{P}^2}.\]
But $\mathfrak{p}$ splits completely in $L/K$, which implies $N_{L/\mathbb{Q}}(\mathfrak{P})=N_{K/\mathbb{Q}}(\mathfrak{p})$, and hence $\mathfrak{p}$ is indeed non-Wieferich with respect to $\eta.$
\end{proof}
\noindent So we can perhaps try to get an a finite extension of $K$ with trivial class group. But this is not a very common phenomenon, and in fact not always true. Golod and Shafarevich in \cite{Gol} showed that for every $n$ there are infinitely many number fields $K$ of degree $n$, which do not have a finite extension with trivial class group. Pollak (see \cite{Pol} pp. 175) has given example of such a number field of degree $2.$\\
\newline
\noindent Instead, we shall invoke the theory of Hilbert class field. So let us note down the necessary tools from this theory.
\begin{definition}[Hilbert class field] \label{Hil} Consider the maximal unramified abelian extension $K'$ of $K$, which contains all other unramified abelian extensions of $K$. This finite field extension $K'$ is called the Hilbert class field of $K$. 
\end{definition}
\noindent It is known that $K'$ is a finite Galois extension of $K$ and $[K' : K]=h_K$, where $h_K$ is the class number of K. In fact, the ideal class group of $K$ is isomorphic to the Galois group of $K'$ over $K.$ It follows from class field theory that every ideal of $\mathcal{O}_K$ extends to a principal ideal of the ring extension $\mathcal{O}_{K'}$.  
\subsection{Siegel's theorem and one application}
\noindent This small subsection is intended to introduce some notations and results from elliptic curves, which will be needed in the last section. Let $E$ be an elliptic curve defined over $K.$ Denote $\Delta_E$ to be the discriminant of $E.$ Let $P \in E(K)$ be a $K-$rational point. We define the weil height of $P$ to be $h(P)=\frac{h(x(P))}{2}$ and the canonical height,
\[\hat{h}(P)=\frac{1}{2}\lim_{n \to \infty}\frac{h(nP)}{n^2}.\]
\noindent Let us now recall a version of Siegel's theorem, which suits our main purpose.
\begin{proposition}[A version of Siegel] \label{Siegel} Fix the point $\infty \in E(\overline{K})$, and an absolute value $v \in M_{K, \infty}$. Then
\[\lim_{P \in E(\overline{K}), h(x(P)) \to \infty} \frac{\log \big(\min \big \{||x(P)||_v,1\big\}\big)}{h(x(P))} \to 0.\]

\end{proposition}
\begin{proof} Follows at once by taking,
\[f=x, Q= \infty,\]
in Theorem 3.1 of \cite{Silverman}.

\end{proof}
\begin{corollary} \label{Siegel corollary} Let $P_n=\frac{a_n}{b_n}$ be an infinite sequence of points in $E(K).$ Then the following holds,
\[N_{K/\mathbb{Q}}(b_n)^{1-\epsilon} \ll_{K,\epsilon} N_{K/\mathbb{Q}}(a_n) \ll_{K, \epsilon} N_{K/\mathbb{Q}}(b_n)^{1+\epsilon}.\]
\end{corollary}

\begin{proof} Siegel's theorem implies,
\[\lim_{n \to \infty} \log \big(\min \big\{||a_n||_v, ||b_n||_v\big\}\big)=0,\]
for any $v \in M_{K, \infty}.$ In particular, for any $v \in M_{K, \infty}$
\[\lim_{n \to \infty} \frac{\log(||a_n||_v)}{\log(||b_n||_v)}=1.\]
And so, 
\[||b_n||_v^{1-\epsilon} \ll_{K, \epsilon} ||a_n||_v \ll_{K, \epsilon} ||b_n||_v^{1+\epsilon}.\]
Multiplying over all $v \in M_{K, \infty},$ we get the desired result.

\end{proof}

\subsection{Growth in cyclotomic polynomials} Let $\phi_n(x)$ be the $n^{th}$ cyclotmic polynomial. It is a polynomial of degree $\phi(n)$ in $\mathbb{Z}[X].$ Then we have the following estimate,
\begin{proposition} \label{Growth} Let $z \in \mathbb{C}$ be a complex number with $|z| \geq 2,$ then 
\[|\phi_n(z)| \geq \frac{1}{3}\Big(\frac{|z|-1}{3}\Big)^{\phi(n)-1}\]
for any $n \geq 1.$

\end{proposition}

\begin{proof} Write $z=e^{i\theta}$ with $x \in \mathbb{R}$ and $0<\theta \leq 2\pi,$ and let $\omega_n$ be the primitive $n^{th}$ root of unity. We can write
\[\frac{\phi_n(z)}{(z+e^{i\theta})^{\phi(n)}}=\sideset{}{'}\prod_{i=1}^{n}\frac{z-\omega_n^i}{z+e^{i\theta}}.\]
In particular, 
\[\Big|\frac{\phi_n(z)}{(z+e^{i\theta})^{\phi(n)}}\Big|=\Big|\sideset{}{'}\prod_{i=1}^{n} \frac{x-e^{i(\frac{2\pi}{n}-\theta)}}{x+1}\Big|,\]
where both of these restricted products are running $i's$ co-prime to $n.$ Now $1-\Big|\frac{x-e^{i(\frac{2\pi}{n}-\theta)}}{x+1}\Big|$ is asymptotic to $\frac{2}{x^2}\big(1+\cos(\frac{2\pi}{n}-\theta)\big).$ Therefore it is positive and decreasing in $x$ and, $1$ otherwise. In particular,
\[\Big|\frac{\phi_n(z)}{(z+e^{i\theta})^{\phi(n)}}\Big| \geq \frac{\phi_n(2)}{3^{\phi(n)}}(|z|-1)^{\phi(n)-1} \geq \frac{1}{3}\Big(\frac{|z|-1}{3}\Big)^{\phi(n)-1}\]
since $|\phi_n(2)| \geq 1.$
\end{proof}

%\begin{proposition} If $h_K$ has prime order $p$ and $h_{K^H}$ has either order $1$ or a prime order $q$ such that that $p \nmid q-1,$ then class group of $K^H$ is trivial. 
%\end{proposition}

%\begin{proposition} If $\text{Gal}(K^H/K)$ has prime order, then class group of $K^H$ is trivial.

%\end{proposition}

%\begin{proof} We know class group of $K^H$ is isomorphic to Hilbert class field of $K^H,$ which is $(K^H)^H$ by notation. First of all $(K^H)^H$ is unramified over $K$. Now let,
%\[G=\text{Gal}\Big((K^H)^H/K\Big), G'=\text{Gal}\Big((K^H)^H/K^H\Big).\]
%It is clear that $G'$ is abelian and $G/G'$ is cyclic group of prime order. Hence by standard group theory, $G'$ is normal in $G$ and 

%\end{proof}

%\\
%\newline
%\noindent Another way is to consider the ring of $S-$ integers. Let us define what it is and mention its some of the properties.

%\begin{definition} Denote,
%\[O_S(K)=\{\alpha \in O_K \mid v(\alpha)>0 \hspace{0.25cm} \forall v \notin S\}.\]
%This $O_S(K)$ is called ring of integers.
%\end{definition}
%\noindent It is known that for large enough $S, O_S(K)$ is a PID. In particular we can actually write any element of $O_K$ uniquely as product of $S$-integers.

%\end{preliminary}

\section{Effective Srinivas-Subramani} 
\noindent First aim of this section was to remove the assumption on class number and a mild condition on units. We shall do both things simultaneously. For all but finitely many units, we may assume there exist $\sigma \in M_{K, \infty}$ such that $1<|\sigma(\eta)|.$ The hypothesis on class number was needed in \cite{Srini} to ensure that we can write $\eta^n-1$ as product of primes uniquely. For arbitrary number fields, we could however write 
$$(\eta^n-1)=U_nV_n $$
where $U_n$ is a square-free and and $V_n$ is a square-full ideal in $\mathcal{O}_K.$ 
\begin{claim} \label{prime tracker} Let $\pi$ be a prime ideal dividing $U_n.$ Then $\pi$ is non-Wieferich prime with respect to $\eta.$
\end{claim}
\begin{proof} The proof is exactly as proof of Lemma 5.1 in \cite{Srini}.

\end{proof}
\noindent We can write,
\[\sigma(\eta)^n=1+\sigma(U_n)\sigma(V_n),\]
and hope to show $N(U_n)=N(\sigma(U_n)) \to \infty$ as done in \cite{Srini}. The hope is not bad because, the assumption on $\eta$ was crucial at this stage, and we are actually doing everything with $\sigma(\eta)$ instead. But the main problem is, we can not apply ABC anymore, since $U_n, V_n$'s are not always elements of $\mathcal{O}_K.$\\ 
\newline
\noindent Recall that we have an extension $K'$ of $K$, of degree $h_K$ such that, all ideals in $\mathcal{O}_K$ are principal in $\mathcal{O}_{K'}.$ In particular, we can factorize $\eta^n-1$ as $u_nv_n$ in $\mathcal{O}_{K'}$ such that $u_n$ and $v_n$ are obtained by lifting $U_n$ and $V_n$ respectively. It is then evident that $v_{\mathfrak{P}}(v_n)$ is even for any prime ideal $\mathfrak{P}$ in $\mathcal{O}_{K'}.$ To show infinitude of non-Wieferich primes in $\mathcal{O}_K,$ it is then enough to show $\{N_{K'/\mathbb{Q}}(u_n)\}$ is unbounded because there are only finitely many prime ideals in $\mathcal{O}_{K'}$ lying below of a prime ideal in $\mathcal{O}_{K}.$
\begin{lemma} \label{norm ineq 1} Following the same notations we have,
\[N_{K'/\mathbb{Q}}(u_n)^{(2\deg(K)h_K-1)(1+\epsilon)} \gg_{K, \epsilon}  |\sigma(\eta)|^{n(1-\epsilon)},\]
for any $\epsilon>0.$
\end{lemma}

\begin{proof} First we start by writing 
\[\sigma(\eta)^n=1+\sigma(u_n)\sigma(v_n).\]
\noindent Modifying equation (13) in \cite{Srini} we have,
\[\prod_{\mathfrak{P} \mid \sigma(u_n)}N(\mathfrak{P})^{v_{\mathfrak{P}}(p)} \leq N_{K'/\mathbb{Q}}(u_n)^{\deg(K)h_K}.\]
The $h_K$ factor is coming because $K'$ has degree $\deg(K)h_K$ over $\mathbb{Q}.$ On the other hand, as we already discussed earlier, the maximum power of $\mathfrak{P}$ dividing $\sigma(v_n)$ is always at least $2$ (if non-zero). Arguing same as in \cite{Srini},
\[\prod_{\mathfrak{P} \mid \sigma(v_n)}N(\mathfrak{P})^{2v_{\mathfrak{P}}(p)}\leq \Big(\sideset{}{'}\prod_{\mathfrak{P} \mid \sigma(u_n)}N(\mathfrak{P})^{2h_K}\Big)\sqrt{N(v_n)}.\]
The restricted product above runs over ramified primes in $K',$ which is clearly finite. And hence,
\[|\sigma(\eta)^n|\leq \Big(N(u_n)^{\deg(K)h_K}\sqrt{N(v_n)}\Big)^{1+\epsilon}.\]
Arguing similarly as in page $7$ in \cite{Srini}, and keeping the condition on $\eta$ on mind we have $N(v_n)< C|\sigma(\eta)|^n|,$ and hence
\[N_{K'/\mathbb{Q}}(u_n)^{(2\deg(K)h_K-1)(1+\epsilon)} \gg  |\sigma(\eta)|^{n(1-\epsilon)},\]
for any $\epsilon>0.$ 
\end{proof}
\noindent We now want to give a lower bound for the number of non-Wieferich primes. First observe that, $\{N_{K'/\mathbb{Q}}(u_n)\}$ is unbounded, so we may assume (after some stage), $1<|N_{K'/\mathbb{Q}}(u_n)|$. It is clear that for any two primes $\pi_1 \neq \pi_2 \in \mathcal{O}_K$ such that $\pi_1,\pi_2$ not lying over a same prime, the quantities $\eta^{N(\pi_1)}-1, \eta^{N(\pi_2)}-1$ can not have a common prime factor because $\mathrm{gcd}(N(\pi_1), N(\pi_2))=1.$ One then needs only to find primes $\pi \in \mathcal{O}_K$ such that
\[ N_{K'/\mathbb{Q}}(\eta^{N(\pi)}-1) \leq x^{h_k},\]
because $N(\pi)^{h_K} \leq N_{K/\mathbb{Q}}(\eta^{N(\pi)}-1)^{h_K}=N_{K'/\mathbb{Q}}(\eta^{N(\pi)}-1).$
By assumption, $|\sigma(\eta)| \leq 1$ for all but one embedding $\sigma \in M_K.$ In particular,
\[N(\eta^{N(\pi)}-1)\leq |\sigma(\eta)|^{N(\pi)}2^d.\]
And so we need primes $\pi \in \mathcal{O}_K$ such that $N(\pi) \ll_{K} \log_{|\sigma(\eta)|} X,$ and it is well known that there are at least $\frac{\log x}{\log \log x}$ many of such. We then get a lower bound of order $\frac{\log x}{\log\log x}.$ However we should aim for $\log x,$ since this was given by Silverman over $\mathbb{Q}.$\\ 
\newline
Let $\phi_n(x)$ be the $n^{\text{th}}$ cyclotomic polynomial. Since $|\eta|>1,$ first of all it is clear that $\phi_n(\eta) \neq 0$. Again, if class group of $K$ is trivial then it is fine, otherwise arguing exactly same as before we can still do the factorization
\[\phi_n(\eta)=u'_nv'_n\]
over $K'.$ The right hand side above make sense because, $\phi_n(\eta)$ divides $\eta^n-1$ in $O_K$ and so $(\phi_n(\eta))=U'_nV_n'$ with $U'_n \mid U_n$ and $V'_n \mid V_n.$ Then one gets $u'_n, v'_n$ by lifting $U'_n, V'_n$ respectively. 
\begin{lemma} \[N_{K'/\mathbb{Q}}(u'_n) \leq N_{K'/\mathbb{Q}}(u_n) \hspace{0.15cm} \text{and} \hspace{0.15cm} N_{K'/\mathbb{Q}}(v'_n) \leq N_{K'/\mathbb{Q}}(v_n).\]
\end{lemma}
\begin{proof} We know $U_n' \mid U_n$ is $\mathcal{O}_K$ and so by definition, $(u'_n) \mid (u_n)$ in $\mathcal{O}_{K'}.$ In particular,
\[N_{K'/\mathbb{Q}}(u'_n)=N\big((u'_n)\big) \leq N\big((u_n)\big) \leq N_{K'/\mathbb{Q}}(u_n).\]
The other part follows similarly.
\end{proof}
\begin{lemma} \label{cycl 1} Let $\pi$ be a prime dividing $u'_n,$ then $\pi$ is non-Wieferich prime in $O_{K'}$ with respect to $\eta.$ If $(n, N(\pi))=1,$ then
\[N_{K'/\mathbb{Q}}(\pi) \equiv 1 \pmod n.\]
In other words, the order of $\eta$ modulo $\pi$ is exactly $n.$

\end{lemma}
\begin{proof} It is already known from \cite{Srini} that $\pi$ definitely is a non-Wieferich prime, because $\pi \mid u'_n \mid \eta^n-1.$ On the other hand,
\[\eta^n-1=\prod_{d \mid n} \phi_d(\eta)\]
If $\pi$ divides $\eta^{d}-1$ for some non-trivial divisor $d$ of $n,$ then $\pi^2|\eta^n-1.$ In other words, the polynomial $x^n-1$ has multiple roots in the residue field of $\pi,$ which have characteristic $p.$ And this is not possible since $(n, N(\pi))=1.$ This implies order of $\eta$ in the residue field of $\pi$ is exactly $n,$ and so does the result follows.
\end{proof}
\noindent From this point we shall assume $\sigma$ to be trivial, because a choice of arbitrary $\sigma$ does not really affect any arguments. We now need to use Proposition \ref{Growth} and for that we eventually need to assume $|\eta|$ is large enough. But we only know $|\eta|>1,$ so we do everything with $\eta'=\eta^M$ for large enough $M$ satisfying the purpose. One can see any arguments will not be affected if we do everything with $\eta'$ instead.
\begin{lemma} \label{norm ineq 2}
\[N_{K'/\mathbb{Q}}(u'_n)\gg_{K, \epsilon} \frac{\Big(\frac{|\eta|-1}{3}\Big)^{\phi(n)-1}}{|\eta|^{\frac{n\epsilon}{1+\epsilon}}}.\]

\end{lemma}

\begin{proof} Doing analogously as in page 7 of \cite{Srini}, we get the following variant of their equation (15)
\begin{equation} \label{eqn 1}
    |\eta|^n \ll_{\epsilon} \Big( N_{K'/\mathbb{Q}}(u_n)^{\deg(K)h_K}\sqrt{N_{K'/\mathbb{Q}}(v_n)}\Big)^{1+\epsilon},
    \end{equation}
where the $h_K$ factor comes because $[K':K]=h_K.$ From $\eta^n=u_nv_n,$ we get $N(u_n)N(v_n) \ll_{K} |\eta|^n.$ In particular, 
\begin{equation} \label{eqn 2}
    N_{K'/\mathbb{Q}}(v'_n)\leq N_{K'/\mathbb{Q}}(v_n) \ll_{K} |\eta|^{\frac{2n\epsilon}{(2\deg(K)h_K-1)(1+\epsilon)}}.
\end{equation}
On the other hand, 
\[N_{K'/\mathbb{Q}}(u'_n)N_{K'/\mathbb{Q}}(v'_n)=N_{K'/\mathbb{Q}}(\phi_n(\eta))=\prod_{\sigma \in M_{K', \infty}}\prod_{i=1}^{\phi(n)} \sigma((\eta-\omega^i)).\]
By Lemma \ref{Ram book}, we have
\[|(\eta^{(j)}-\omega^i)| \leq 2, \hspace{0.1cm} \text{for all} \hspace{0.1cm} 2 \leq j \leq \deg(K'), \hspace{0.05cm} 1 \leq i\leq \phi(n).\]
In particular we then have, 
\begin{equation} \label{eqn 3}
    N_{K'/\mathbb{Q}}(v'_n)N_{K'/\mathbb{Q}}(u'_n)= N(\phi_n(\eta)) \gg_{K} \frac{1}{3}\Big(\frac{|\eta|-1}{3}\Big)^{\phi(n)-1}.
    \end{equation}
Now combining equations (\ref{eqn 2}) and (\ref{eqn 3}) we finally have \[N_{K'/\mathbb{Q}}(u'_n) \gg_{K,\epsilon} \frac{\Big(\frac{|\eta|-1}{3}\Big)^{\phi(n)-1}}{|\eta|^{\frac{n\epsilon}{1+\epsilon}}}.\]
%since $|\phi_n(\eta)| \gg_{K}\eta|^{\phi(n)}.$ The last inequality is true since
%\[|\eta-\omega^i| > c|\eta|,\]
%is true for any $i$ with a constant $c>0$ depending only on $K.$
\end{proof}

\begin{lemma} \label{imp estimate} Denote $\mathcal{S}(X)$ to be the number of non-Wieferich prime ideals in $\mathcal{O}_K$ of norm at most $x,$ with respect to base $\eta.$ Then we have the following estimate,
\[ |\mathcal{S}(X)| \geq |\{n \leq \log_{|\eta|}x \mid |N(U'_n)|>n^{\deg(K)}\}|.\]
\end{lemma} 
\begin{proof} For all such $n$'s in the right hand side, we have
\[N(u'_n)N(v'_n)=N(\phi_n(\eta)) \ll |\eta|^n \leq x.\]
On the other hand, $N(U'_n)\leq N(u'_n)N(v'_n) \leq x.$ Let $\pi_n$ be a prime divisor of $U'_n,$ and hence $N(\pi_n) \leq x.$\\
\newline
\noindent Now we take $\pi_n$ to be a prime dividing $U'_n$ such that $(N(\pi_n),n)=1.$ We can do this because $N(U_n')>n^{\deg(K)}.$ Lemma \ref{cycl 1} completes the proof because $\pi_n$'s are now pairwise distinct.
\end{proof}
\noindent We are finally done with all preparations for the lower bound.
\begin{proof}[\textbf{Proof of Theorem 1.0.2}] From lemma \ref{norm ineq 2} we have,
\[N_{K'/\mathbb{Q}}(u'_n)\gg_{K, \epsilon} \frac{\Big(\frac{|\eta|-1}{3}\Big)^{\phi(n)-1}}{|\eta|^{\frac{n\epsilon}{1+\epsilon}}}.\]
On the hand for any $\eta$ with $|\eta|-1 \geq 3e,$
\[N(U'_n)^{h_K} \geq N_{K'/\mathbb{Q}}(u'_n) \gg_{K,\epsilon} \frac{\Big(\frac{|\eta|-1}{3}\Big)^{\phi(n)-1}}{|\eta|^{\frac{n\epsilon}{1+\epsilon}}} \gg_{K, \epsilon} n^{\deg(K)}\]
holds if $\phi(n) \geq \epsilon n +C_{K, \epsilon}$
for a positive constant $C_{K, \epsilon}$ not depending on $n.$ Now one can prove the desired result combining Lemma \ref{imp estimate}, Lemma \ref{cycl 1} and Lemma 6 of \cite{Sil}.
\end{proof}
\section{Generalization to primes in congruence classes}

\noindent Consider rational primes $p$, which are congruent to $1$ modulo $k$. It was first proved by Graves-Murty in \cite{Murty} that there are infinitely many non-Wieferich primes of such a form, with a lower bound of order $\frac{\log x}{\log \log x}.$ Their idea was roughly to study squarefree parts of $\phi_{nk}(a)$. The main point is, if $p\mid \phi_{nk}(a)$ then either $p \mid nk$ or $p \equiv 1 \pmod {nk}.$ Later, Ding and Cheng improved that bound in \cite{Chen}. They established a lower bound of order $\frac{\log x}{\log \log x}(\log \log \log x)^M$ for any natural number $M.$ \\
\newline
\noindent Analogously, consider primes $\pi \in \mathcal{O}_K$ of form $N(\pi) \equiv 1 \pmod k.$ As a generalization of \cite{Murty} and \cite{Chen}, one may ask whether there are infinitely many non-Wieferich primes $\pi$ of the same form or not. In this section, we first give an affirmative answer to that question, removing the $(\log \log \log x)^M$ term over any arbitrary number fields. 
\begin{proof}[\textbf{Proof of Theorem 1.0.3}] First we denote $\mathcal{S}_k(X)$ to be the number of non-Wieferich primes (with respect to unit $\eta$) in $\mathcal{O}_K$ of norm is at most $x$ and congruent to $1$ modulo $k.$ From Lemma \ref{norm ineq 1} we have,
\[N_{K'/\mathbb{Q}}(u'_{nk})\gg_{K, \epsilon} \frac{\Big(\frac{|\eta|-1}{3}\Big)^{\phi(nk)-1}}{|\eta|^{\frac{nk\epsilon}{1+\epsilon}}}.\]
From Lemma \ref{cycl 1}, if $\pi$ is a prime dividing $u'_{nk}$ and $(nk, N(\pi))=1,$ then
\[N(\pi) \equiv 1 \pmod{nk}.\]
Lemma \ref{imp estimate} gives,
\[S_k(X) \geq |\{nk \leq \log_{|\eta|}(x) \mid N(U'_{nk})>(nk)^{\deg(K)}\}|.\]
Following the proof of Theorem 1.0.2, we only need to count 
\[\{nk \leq Y \mid \phi(nk) \geq \epsilon nk\}.\]
Note that $\phi(nk)=\phi(n)\phi(k)\frac{d}{\phi(d)}$ where $d=(n,k).$ Basically we then need to count,
\[\Big\{n \leq \frac{Y}{k} \mid \phi(n) \geq \epsilon \frac{k}{\phi(k)}\frac{\phi(d)}{d} n\Big\}.\]
Since $k$ is fixed, and all those $d$'s divide $k$, hence the set of numbers $\{\frac{k}{\phi(k)}\frac{\phi(d)}{d}\}$ is finite. For small enough $\epsilon$'s it is then clear that 
\[\{nk \leq Y \mid \phi(nk) \geq \epsilon nk\} \gg \frac{Y}{k}.\]
Now we can complete proof the theorem using Lemma 6 of \cite{Sil}.
\end{proof}
%Arguing same as in proof of Theorem 1.0.2, we obtain
%\begin{corollary} Let $K$ be an arbitrary number field and assume ABC holds for %$K.$ Then there are infinitely many non-Wieferich primes in $O_K'$
%\end{corollary}

\section{in other algebraic groups}
\noindent Continuing our discussion from the last paragraph of preliminary, we first generalize the work of K\"uhn and M\"uller. Suppose $E$ be an elliptic curve defined over $K$ and $P \in E(K)$ be a point of infinite order. Let $\pi$ be a prime ideal in $\mathcal{O}_K$ and $N_{\pi}=|E \pmod{\pi}|$. Now the question is whether $N_{\pi}P \equiv 0 \pmod {\pi^2}.$ If we are in trivial class group case, any point $P \in E(K)$ can be written as $\Big(\frac{a_P}{d^2_P},\frac{b_P}{d^3_P}\Big)$ uniquely. If not, we could still write
\[(x_P)=\frac{\mathfrak{a}_P}{\mathfrak{d}^2_P}\hspace{0.15cm}\text{and}\hspace{0.15cm}(y_P)=\frac{\mathfrak{b}_P}{\mathfrak{d}^3_P},\]
where $\mathfrak{a}_P, \mathfrak{b}_P, \mathfrak{d}_P$ are ideals in $\mathcal{O}_K.$ We can then consider $d_P, a_P, b_P$ to be the lifts of $\mathfrak{d}_P, \mathfrak{a}_P$ and $\mathfrak{b}_P$ respectively to $K',$ and argue similarly as in the previous section. So let us do everything with trivial class group. To fix notations once and for all, write
\[P=\Big(\frac{a_P}{d^2_P}, \frac{b_P}{d^3_P}\Big).\]
Following Silverman's approach in \cite{Sil}, fix $P \in E(K)$ a non-torsion point and write $nP=\Big(\frac{a_n}{d^2_n}, \frac{b_n}{d^3_n}\Big).$ Since we are working over trivial class group, we can write 
$$d_n=u_nv_n,$$
where $u_n$ is the square-free and $v_n$ is the square-full part of $d_n.$ We further denote $D_n$ to be the greatest divisor of $d_n$ not dividing $d_1d_2 \cdots d_{n-1}.$ We further write $D_n=U_nV_n$ similarly as before. We can do all of these because the class group is trivial. If not, we could still define $D_n$ as an ideal in $\mathcal{O}_K$ and work with its lift to $K'.$\\
\newline
\noindent Let us first introduce to the uniform ABC conjecture for curves over number fields, as proposed by Vojta.
\begin{definition} Let $X$ be a smooth, proper, geometrically
connected curve over a number field $K$. Let $D \subset X$ be an effective reduced divisor, and $\omega_X$ be the canonical sheaf on $X$. Fix a proper regular model $\mathfrak{X}$ of $X$ over $\mathrm{Spec}(\mathcal{O}_K)$ and extend $D$ to an effective horizontal divisor $\mathfrak{D}$ on $\mathfrak{X}$. For any point $P \in X(K),$ define
\[\mathrm{cond}_{\mathfrak{X}, \mathfrak{D}}(P)= \prod_{p \in S} N(\mathfrak{p})^{\frac{v_{\mathfrak{p}}(p)}{[K:\mathbb{Q}]}},\]
where $S$ is the set of finite primes $\mathfrak{p}$ of $K$ such that the intersection multiplicity $(\mathfrak{P}, \mathfrak{D})_{\mathfrak{p}} \neq 0.$
\end{definition}

\begin{conjecture}[ABC for curves]\label{ABC for curves} Suppose that $\omega_X(D)$ is ample and $h_{\omega_X}(D)$ be a Weil height function on X with respect to $\omega_X(D)$. Then for any $ \epsilon > 0$ and $d \in \mathbb{N}$, there exists a constant $c = c(\epsilon, d, X , D)$ such that
\[h_{\omega_X(D)}(P) \leq (1 + \epsilon) \big(\log \mathrm{disc}(k(P)\big) + \log \mathrm{cond}_{\mathfrak{X},\mathfrak{D}}(P) + c\]
for all $P \in X(K)-\mathrm{supp}(D)$ satisfying $[k(P) : \mathbb{Q}] \leq d,$ where $k(P)$ is the residue field at $P.$
\end{conjecture}
\noindent Interestingly we then have,
\begin{proposition} \label{equiv} ABC for number fields is equivalent to $\mathrm{Conjecture} 2.$

\end{proposition}

\begin{proof} We first show that Conjecture 2 above implies ABC for number fields. Let $X=\mathbb{P}^1$ and take $a,b,c \in K$ with $a=b+c$ Consider the point $P=[a:c] \in X(K).$ Let $D=(0)+(1)+(\infty) \in \text{Div}(X)$ be an effective divisor with $\deg(\omega_X(D))=2g(X)-2+\deg(D)=1.$ In particular, $\omega_X(D)$ is ample of degree $1.$ Therefore, up to a bounded constant $h_{\omega_X(D)}$ is the weil height $h.$ Now the conjecture above implies,
\[H(a,b,c) \ll_{K} H(P) \ll_{\epsilon} \Big(\text{cond}_{\mathfrak{X}, \mathfrak{D}} (P)\Big)^{1+\epsilon},\]
where the first implication is coming from Remark 2.1.1. For any prime $\mathfrak{p},$ note that $\big(\mathfrak{P}, (0)\big)_{\mathfrak{p}} \neq 0,  \big(\mathfrak{P}, (\infty)\big)_{\mathfrak{p}} \neq 0$ and $\big(\mathfrak{P}, (1)\big)_{\mathfrak{p}} \neq 0$ if and only if $\mathfrak{p} \mid a, \mathfrak{p} \mid b$ and $\mathfrak{p} \mid c=a-b$ holds respectively. In particular,
\[\text{cond}_{\mathfrak{X}, \mathfrak{D}} (P)= \Big( \prod_{\mathfrak{p} \mid abc}N_{K/\mathbb{Q}}(\mathfrak{p})^{v_{\mathfrak{p}}(p)}\Big)^{\frac{1}{[K:\mathbb{Q}]}}=\text{rad}(abc),\]
and that completes the proof for one direction. For the converse, see Theorem 2.1 of \cite{Moc}. 
\end{proof}
\noindent Let us now prove the main result of this section. But we first need some key lemmas.
\begin{lemma} \label{norm ineq 4} Following the previously introduced notations,
\[\log N_{K/\mathbb{Q}}(v_P) \ll_{\epsilon} \epsilon \log N_{K/\mathbb{Q}}(d_P)+c.\]

\end{lemma}
\begin{proof} Consider $X=E$ and $D=(0) \in \text{Div}(X).$ Hence $\omega_X(D)$ has degree $1$ and so is ample. Therefore, $h_{\omega_X(D)}$ is the Weil height $h$ on $E(K)$ up to a bounded constant. On the other hand,
\[\text{cond}_{\mathfrak{X}, \mathfrak{D}}(P)=\prod_{\mathfrak{p}\in S} N(\mathfrak{p})^{\frac{v_{\mathfrak{p}}}{[K:\mathbb{Q}]}},\]
where $S$ is the set of primes $\mathfrak{p}$ with $\big(\mathfrak{P}, (0)\big)_{\mathfrak{p}} \neq 0.$ If $E$ has good reduction at $\mathfrak{p},$ then the previous statement holds if and only if $\mathfrak{p} \mid d_P.$ By Conjecture 2, we then get
\begin{equation} \label{11}
h(P) \leq (1+\epsilon) \log (\text{rad}(d_P))+c_E.
\end{equation}
Note that 
\begin{equation}\label{2}
\sum_{v \in M_K} \log \max\Big\{||\frac{a_P}{d^2_P}||_v, 1\Big\} \geq \sum_{v \in M'_K} \log \max\Big\{||\frac{a_P}{d^2_P}||_v, 1\Big\} \geq -2\sum_{v\mid d_P} \log \min\Big\{||d_P||_v, 1\Big\}.
\end{equation}
\noindent On the other hand,
\begin{equation}\label{3}
-2\sum_{v\mid d_P} \log \min\Big\{||d_P||_v, 1\Big\}=2\log N_{K/\mathbb{Q}}(d_P)=2\big(\log N_{K/\mathbb{Q}}(u_p)+ \log N_{K/\mathbb{Q}}(v_P)\big).
\end{equation}
and 
\begin{equation}\label{4}
\log(\mathrm{rad}(d_P)) \leq N_{K/\mathbb{Q}}(u_P)+ \frac{1}{2} N_{K/\mathbb{Q}}(v_P).
\end{equation}
Combining (\ref{11}), (\ref{2}), (\ref{3}) and (\ref{4}) we get
\[\frac{1-\epsilon}{2} \log N_{K/\mathbb{Q}}(v_P) \leq \epsilon \log N_{K/\mathbb{Q}}(u_P)+c',\]
and this finishes proof of the lemma since $d_P=u_Pv_P.$
\end{proof}
\noindent We have already discussed about why we can afford to be in trivial class group case. Following the notations in \cite{Sil}, we define $u_n, v_n, U_n, V_n$ analogously. 
\begin{lemma} \label{track 2} Let $\mathfrak{p}$ be a prime ideal in $\mathcal{O}_K$ dividing $U_n$ and not dividing $d_2\Delta_E,$
then
\[m_{\mathfrak{p}}=n, N_{\mathfrak{p}}P \neq 0 \pmod{\mathfrak{p}^2},\]
where $m_{\mathfrak{p}}$ is the least number $n$ such that $nP \equiv 0 \pmod {\mathfrak{p}}.$
\end{lemma}
\begin{proof} Let $\mathcal{F}_{\mathfrak{p}}, \mathcal{F}_{\mathfrak{p}^2}$ be the formal groups of $E$ at $\mathfrak{p}$ and $\mathfrak{p}^2$ respectively. Now the proof is exactly analogous to Lemma 11 in \cite{Sil}. Because,
\[\mathcal{F}_{\mathfrak{p}}/\mathcal{F}_{\mathfrak{p}^2} \sim k(\mathfrak{p}),\]
where $k(\mathfrak{p})$ is the residue field at $\mathfrak{p},$ and Hasse-Weil gives $N_{\mathfrak{p}} \leq (\sqrt{k(\mathfrak{p})}+1)^2$. The rest is exactly same.
\end{proof}

\begin{lemma} \label{norm height 1} We have
\[n^2 \hat{h}(P) \geq \log N_{K/\mathbb{Q}}(d_n) \gg_{E, \epsilon} (1-\epsilon)n^2 \hat{h}(P).\]

\end{lemma}

\begin{proof} After using Corollary \ref{Siegel} from preliminary, the proof is exactly same as Lemma 8 of \cite{Sil}.  

\end{proof}
\noindent Furthermore, we have the following crucial lower bound,
\begin{lemma} \label{norm height 2}
\[\log N_{K/\mathbb{Q}}(D_n) \gg_{E,\epsilon} \big( \frac{1}{3}-\epsilon)n^2\hat{h}(P)-\log n.\]

\end{lemma}

\begin{proof} First part of Silverman's arguments in the proof of Lemma 9 (in \cite{Sil}) uses some facts from formal group of $E.$ All of those carries over number fields because the main required fact was Proposition VII.2.2 of \cite{Silverman}, and that is valid for any DVR. On the other hand second part of the proof, i.e. about estimating $\log N(D_n)$ follows immediately from the previous lemma. 

\end{proof}
\begin{proof}[\textbf{Proof of Theorem 1.0.4}] From Lemma \ref{norm ineq 4}, Lemma \ref{norm height 1} and Lemma \ref{norm height 2} we obtain,

\[\log N_{K/\mathbb{Q}}(U_n)= \log N_{K/\mathbb{Q}}(D_n)-\log N_{K/\mathbb{Q}} (V_n) \gg_{E, \epsilon} \big(\frac{1}{3}-\epsilon\big)n^2\hat{h}(P)-\log n.\]
Sine $P$ is non-torsion from the beginning, $\hat{h}(P) \neq 0.$ And so, we may assume 
\[N_{K/\mathbb{Q}}(U_n)>N_{K/\mathbb{Q}}(d_2\Delta_E),\] 
for all but finitely many $n$. For all such $n's,$ we can therefore pick a prime ideal $\pi_n$ dividing $U_n$ co prime to $d_2\Delta_E.$ Lemma \ref{track 2} then shows $\pi_n$ is a non-Wieferich prime for $P.$ On the other hand, we have
\[\log N_{K/\mathbb{Q}}(U_n) \leq \log N_{K/\mathbb{Q}}(D_n) \ll_{E} n^2 \hat{h}(P).\]
So for all $n \ll_{E,P} \sqrt{\log x}, \log N_{K/\mathbb{Q}}(U_n) \leq x$. In particular $\pi_n$ is a non-Wiefrich for $P$ with norm at most $x.$ Once again Lemma \ref{track 2} shows, these $\pi_n$'s are pairwise different, and hence the proof is now completed.

\end{proof}
%\noindent As immediate consequences, we get the following applications on other algebraic groups 

%\begin{corollary} Let $A$ be a $g$ dimensional abelian variety over $K$ given by product of $g-$ elliptic curves over $K$. Let $P=(P_1, P_2, \cdots, P_g) \in A(K)$ be a point of infinite order. Then there are at least $\gg_{P} (\log x)^{\frac{g}{2}}$ many non-Wieferich primes with respect to $P.$
% \end{corollary}

%\begin{proof} This is immediate since, 
%\[N_{\mathfrak{p}}=|A(\mathbb{F}_{\mathfrak{p}})|=\prod_{i=1}^{g} |E_i(\mathbb{F}_{\mathfrak{p}})|,\]
%and $N_{\mathfrak{p}}P = 0 \pmod {\mathfrak{p}^2}$ if and only if all of $N_{i,\mathfrak{p}}P = 0 \pmod {\mathfrak{p}^2}$ for all $1 \leq i\leq g$ holds.
%\end{proof}

\section{Acknowledgements}
\noindent I first thank the Mathematics Institute of Georg-August-Universit\"at G\"ottingen for providing a beautiful environment for study and research. I am especially grateful to Prof. K. Srinivas for introducing me to his previous work with M. Subramani. The project was started while visiting the Institute of Mathematical Sciences. I am very grateful for their encouragement and fruitful discussions. I am also indebted to my supervisor Harald Andres Helfgott and Jitendra Bajpai for many helpful discussions. In all my needs, I always found them as kind and helpful persons. I got the idea for proving Proposition \ref{Growth} from an answer posted by Semiclassical in Stack Exchange. I thank the anonymous person for proving with real values, which inspired my work. I also thank to Sulakhana for helping out in proof reading.

\end{document}